\documentclass[12pt]{iopart}

\usepackage{graphicx,epsfig,latexsym,enumerate,iopams,setstack}
\usepackage{pstricks,pst-grad,pst-plot}

\newtheorem{theorem}{Theorem}[section]
\newtheorem{remark}[theorem]{Remark}
\newtheorem{assumption}[theorem]{Assumption}
\newtheorem{example}[theorem]{Example}
\newtheorem{definition}[theorem]{Definition}
\newenvironment{proof}[1][Proof]{\begin{trivlist}
\item[\hskip \labelsep {\bfseries #1}]}{\end{trivlist}}
\eqnobysec

\newcommand{\dd}{\mathrm{d}}

\newcommand{\R}{\mathbb{R}}

\newcommand{\half}{\mbox{$\frac{1}{2}$}}

\newcommand{\ignore}[1]{}

\newcommand{\notmid}{\mid\kern-0.5em\not\kern0.5em}

\newcommand{\dst}{\displaystyle}

\newcommand{\ba}{\begin{array}}
\newcommand{\ea}{\end{array}}

\begin{document}
\title{Local Analysis of Inverse Problems: H\"{o}lder Stability
       and Iterative Reconstruction}

\author{Maarten V. de Hoop$^1$, Lingyun Qiu$^1$ and Otmar Scherzer$^2$}

\address{$^1$ Center for Computational and Applied Mathemematics, Purdue University, West Lafayette, IN 47907}
\address{$^2$ Computational Science Center, University of Vienna, Nordbergstra\ss{}e 15, A-1090 Vienna, Austria
           and
           Johann Radon Institute for Computational and Applied Mathematics (RICAM),
           Austrian Academy of Sciences, Altenbergerstra\ss{}e 69,
           A-4040 Linz, Austria}
\ead{\mailto{mdehoop@purdue.edu}, \mailto{qiu@purdue.edu}, \mailto{otmar.scherzer@univie.ac.at}}



\begin{abstract}
We consider a class of inverse problems defined by a nonlinear mapping
from parameter or model functions to the data, where the inverse mapping
is H\"older continuous with respect to appropriate Banach spaces. We
analyze a nonlinear
Landweber iteration and prove local convergence and convergence rates
with respect to an appropriate distance measure. Opposed to the standard
analysis of the nonlinear Landweber iteration, we do not assume source and
non-linearity conditions, but this analysis is based solely on the
H\"older continuity of the inverse mapping.
\end{abstract}

\submitto{\IP}
\maketitle




\section{Introduction}
\label{sec:1}
In this paper, we study the convergence of certain nonlinear iterative
reconstruction methods for inverse problems in Banach spaces. We consider a class of inverse problems defined by a nonlinear map
from parameter or model functions to the data. The parameter functions
and data are contained in certain Banach spaces, or Hilbert spaces,
respectively.
We explicitly construct sequences of parameter functions by a Landweber
iteration.
Our analysis pertains to obtaining natural conditions for the strong
convergence of these sequences (locally) to the solutions in an
appropriate distance measure.

Our main result establishes convergence of the Landweber iteration if
the inverse problem ensures a H\"{o}lder stability estimate. Moreover,
we prove monotonicity of the residuals defined by the sequence induced
by the iteration. We also obtain the convergence rates without
so-called source and nonlinearity conditions. The stability
condition is a
natural one in the framework of iterative reconstruction.

Extensive research has been carried out to study convergence of the
Landweber iteration \cite{Landweber1951} and its modifications. In the
case of model and data spaces being Hilbert, see Hanke, Neubauer \& Scherzer
\cite{Hanke1995}. An overview of iterative methods for inverse
problems in Hilbert spaces can be found, for example, in Kaltenbacher,
Neubauer \& Scherzer \cite{Kaltenbacher2008}. Sch\"{o}pfer, Louis \&
Schuster \cite{Schopfer2006} presented a nonlinear extension of the
Landweber method to Banach spaces using duality mappings. We use this
iterative method in the analysis presented here.
Duality mappings also play a role in iterative schemes for monotone and
accretive operators
(see Alber \cite{Alber1996}, Chidume \& Zegeye \cite{Chidume2001} and
Zeidler \cite{Zeidler1990,Zeidler1990a}). The model space needs to be
smooth and uniformly convex, however, the data space can be an
arbitrary Banach space. Due to the geometrical characteristics of
Banach spaces other than Hilbert spaces, it is more appropriate to use
Bregman distances rather than Ljapunov functionals to prove
convergence (Osher \textit{et al.}  \cite{Osher2005}). For convergence
rates, see Hofmann \textit{et al.}  \cite{Hofmann2007}. Sch\"{o}pfer,
Louis \& Schuster \cite{Schopfer2008} furthermore considered the
solution of convex split feasibility problems in Banach spaces by
cyclic projections. Under the so-called tangential cone condition,
pertaining to the nonlinear map modelling the data, convergence has
been established; invoking a source condition in a convergence
rate result. Here, we build on the work of Kaltenbacher, Sch\"{o}pfer and
Schuster \cite{Kaltenbacher2009} and revisit these conditions with a
view to stability properties of the inverse problem.

In many inverse problems one probes a medium, or an obstacle, with a
particular type of field and measures the response. From these
measurements one aims to determine the medium properties and/or
(geometrical) structure. Typically, the physical phenomenon is modeled
by partial differential equations and the medium properties by
variable, and possibly singular, coefficients. The interaction of
fields is usually restricted to a bounded domain with
boundary. Experiments can be carried out on the boundary. The goal is
thus to infer information on the coefficients in the interior of the
domain from the associated boundary measurements. The map, solving the
partial differential equations, from coefficients or parameter
functions to the measurements or data is nonlinear. Its injectivity is
studied in the analysis of inverse problems. As an example, we discuss
Electrical Impedance Tomography, where the Dirichlet-to-Neumann map
represents the data, and summarize the conditions leading to Lipschitz
stability.

Traditionally, the Landweber iteration has been viewed as a
fixed-point iteration. However, in general, for inverse problems, the
underlying fixed point
operator is not a contraction. There is an extensive literature of
iterative methods for approximating fixed points of non-expansive
operators. Hanke, Neubauer \& Scherzer \cite{Hanke1995} replace the
condition of non-expansive to a local tangential cone condition, which
guarantees a local result. In the finite-dimensional setting, in
which, for example the model space is $\R^n$, non-convex constraint
optimization problems admitting iterative solutions have been studied
by Curtis \textit{et al.} \cite{Curtis2010}. Under certain
assumptions, they obtain convergence to stationary points of the
associated feasibility problem. In the context of inverse problems
defined by partial differential equations, this setting is motivated
by discretizing the problems prior to studying the convergence
(locally) of the iterations. Inequality constraints are necessary to
enforce locality. The non-convexity is addressed by Hessian
modifications based on inertia tests.

The paper is organized as follows. In the next section, we summarize
certain geometrical aspects of Banach spaces, including (uniform)
smoothness and (uniform) convexity, and their connection to duality
mappings. Smoothness is naturally related to G\^{a}teaux
differentiability. We also introduce the Bregman distance. We then
define the nonlinear Landweber iteration in Banach spaces. In
Section~\ref{sec:3} we introduce the basic assumptions including
H\"{o}lder stability and analyze the convergence of the Landweber
iteration in Hilbert spaces. In Section~\ref{sec:4} we adapt these
assumptions and generalize the analysis of convergence of the
Landweber iteration to Banach spaces. We also establish the
convergence rates. In Section~\ref{sec:6} we give an example, namely,
the reconstruction of conductivity in Electrical Impedance Tomography,
and show that our assumptions can be satisfied.

\section{Landweber iteration in Banach spaces}
\label{sec:2}

Let $X$ and $Y$ be both real Banach spaces. We consider the nonlinear
operator equation
\begin{equation}\label{forward-equ}
   F(x) = y , \quad x \in \mathcal{D}(F) ,\ y \in Y ,
\end{equation}
with domain $\mathcal{D}(F) \subset X$. In applications, $F
:\ \mathcal{D}(F) \rightarrow Y$ models the data. In the inverse
problem one is concerned with the question whether $y$ determines $x$.
We assume that $F$ is continuous, and that $F$ is Fr\'{e}chet
differentiable, locally.

We couple the uniqueness and stability analysis of the inverse problem
to a local solution construction based on the Landweber
iteration. Throughout this paper, we assume that the data $y$ in
(\ref{forward-equ}) is attainable, that is, that (\ref{forward-equ})
has a solution $x^\dag$ (which needs not to be unique).

\subsection{Duality mappings}
\label{sec:2-1}

The duals of $X$ and $Y$ are denoted by $X^*$ and $Y^*$,
respectively. Their norms are denoted by $\|\cdot\|$. We denote the
space of continuous linear operators $X \rightarrow Y$ by
$\mathcal{L}(X,Y)$. Let $A : \mathcal{D}(A) \subset X \rightarrow Y $
be continuous. Here $\mathcal{D}(A)$ denotes the domain of $A$. Let $h
\in \mathcal{D}(A)$ and $k \in X$ and assume that $h + t(k-h) \in
\mathcal{D}(A)$ for all $t\in (0, t_0)$ for some $t_0$, then we denote
by $DA(h)(k)$ the directional derivative of $A$ at $h\in
\mathcal{D}(A)$ in direction $k\in \mathcal{D}(A)$, that is,
\[
DA(h)(k):= \lim_{t \rightarrow 0^+} \frac{A(h+tk) - A(h)}{t}.
\]
If $DA(h) \in \mathcal{L}(X,Y)$, then $DA(h)$ is called G\^{a}teaux differentiable. If, in addition, the convergence is uniform for all $k\in B_{t_0}$, then $DA$ is Fr\'{e}chet differentiable at $h$. For $x \in X$ and $x^* \in X^*$, we write the dual
pair as $\langle x,x^* \rangle = x^*(x)$. We
write $A^*$ for the dual operator $A^* \in \mathcal{L}(Y^*,X^*)$ and
$\|A\| = \|A^*\|$ for the operator norm of $A$. We let $1 < p,q <
\infty$ be conjugate exponents, that is,
\begin{equation}
   \frac{1}{p} + \frac{1}{q} = 1 .
\end{equation}

For $p > 1$, the subdifferential mapping $J_p = \partial f_p : X
\rightarrow 2^{X^*}$ of the convex functional $f_p: x\mapsto
\frac{1}{p}\|x\|^p$ defined by
\begin{equation}\label{Duality-map-def}
   J_p(x) = \{ x^* \in X^* \mid \langle x, x^* \rangle
           = \|x\| \,\|x^*\| \mbox{ and } \|x^*\| = \|x\|^{p-1}\}
\end{equation}
is called the duality mapping of $X$ with gauge function $t \mapsto
t^{p-1}$. Generally, the duality mapping is set-valued. In order to
let $J_p$ be single valued, we need to introduce the notion of
convexity and smoothness of Banach spaces.

One defines the convexity modulus $\delta_X$ of $X$ by
\begin{equation}\label{def:convex-module}
   \delta_X(\epsilon) = \inf_{ x, \tilde{x} \in X} \{
      1 - \| \half (x + \tilde{x})\| \mid \|x\| = \|\tilde{x}\| = 1
            \mbox{ and }\|x - \tilde{x}\| \ge \epsilon\}
\end{equation}
and the smoothness modulus $\rho_X$ of $X$ by
\begin{equation}\label{def:smooth-module}
   \rho_X(\tau) = \sup_{x, \tilde{x} \in X} \{
      \half (\|x + \tau \tilde{x}\| + \|x - \tau \tilde{x}\|
           - 2) \mid \|x\| = \|\tilde{x}\| = 1 \} .
\end{equation}

\medskip\medskip

\begin{definition}
A Banach space $X$ is said to be
\begin{enumerate}[(a)]
\item uniformly convex if $\delta_X(\epsilon) > 0$ for any $\epsilon
  \in (0,2],$
\item uniformly smooth if $\lim_{\tau \rightarrow 0}
  \frac{\rho_X(\tau)}{\tau} =0$,
\item convex of power type $p$ or $p$-convex if there exists a
  constant $C > 0$ such that $\delta_X(\epsilon) \ge C \epsilon^p$,
\item smooth of power type $q$ or $q$-smooth if there exists a
  constant $C > 0$ such that $\rho_X(\tau) \le C \tau^q$.
\end{enumerate}
\end{definition}

\medskip\medskip

\begin{example}
\begin{enumerate}[(a)]
\item A Hilbert space $X$ is $2$-convex and $2$-smooth and $J_2: X
  \rightarrow X$ is the identity mapping.
\item Let $\Omega\subset \mathbb{R}^n$ be an open domain. The Banach
  space $L^p = L^p(\Omega)$, $p > 1$ is uniformly convex and
  uniformly smooth, and
\begin{equation*}
    \delta_{L^p}(\epsilon) \simeq \left\{
    \begin{array}{rl}
    \epsilon^2 , & \quad  1 < p < 2,\\
    \epsilon^p , & \quad  2 \le p < \infty;\\
    \end{array}
    \right.
\end{equation*}
\begin{equation*}
    \rho_{L^p}(\tau) \simeq \left\{
    \begin{array}{rl}
    \tau^p , & \quad  1 < p < 2,\\
    \tau^2 , & \quad  2 \le p < \infty.\\
    \end{array}
    \right.
\end{equation*}
\item For $X = L^r(\mathbb{R}^n)$, $r>1$, we have
\begin{equation*}
\begin{array}{rl}
   J_p : L^r(\mathbb{R}^n) & \rightarrow L^s(\mathbb{R}^n)
\\[0.2cm]
         u(x) &\mapsto \|u\|_{L^r}^{p-r} |u(x)|^{r-2} u(x) ,
\end{array}
\end{equation*}
where $\frac{1}{r} + \frac{1}{s} = 1$.
\end{enumerate}
\end{example}

\medskip\medskip

For a detailed introduction to the geometry of Banach spaces and the
duality mapping, we refer to \cite{Cioranescu1990,Schopfer2006}. We
list the properties we need here in the following theorem.

\medskip\medskip

\begin{theorem} Let $p>1$. The following statements hold true:
\begin{enumerate}[(a)]
\item For every $x \in X$, the set $J_p(x)$ is not empty and it is
  convex and weakly closed in $X^*$.
\item If a Banach space is uniformly convex, it is reflexive.
\item A Banach space $X$ is uniformly convex (resp. uniformly smooth)
  iff $X^*$ is uniformly smooth (resp. uniformly convex).
\item If a Banach space $X$ is uniformly smooth, $J_p(x)$ is single
  valued for all $x \in X$.
\item If a Banach space $X$ is uniformly smooth and uniformly convex,
  $J_p(x)$ is bijective and its inverse $J_p^{-1}: X^* \rightarrow X$
  is given by $J_p^{-1} = J_q^*$ with $J_q^*$ being the duality
  mapping of $X^*$ with gauge function $t \mapsto t^{q-1}$, where $1 <
  p,q < \infty$ are conjugate exponents.
\end{enumerate}
\end{theorem}

\medskip\medskip

Throughout this paper, we assume that $X$ is $p$-convex and $q$-smooth
with $p,q > 1$, hence it is uniformly smooth and uniformly
convex. Furthermore, $X$ is reflexive and its dual $X^*$ has the same
properties. $Y$ is allowed to be an arbitrary Banach space; $j_p$ will
be a single-valued selection of the possibly set-valued duality
mapping of $Y$ with gauge function $t \mapsto t^{p-1}$, $p >
1$. Further restrictions on $X$ and $Y$ will be indicated in the
respective theorems below.

\subsection{Bregman distances}
\label{sec:2-3}

Due to the geometrical characteristics of Banach spaces different from
those of Hilbert spaces, it is often more appropriate to use the
Bregman distance instead of the conventional norm-based functionals
$\|x-\tilde{x}\|^p$ or $\|J_p(x) - J_p(\tilde{x})\|^p$ for convergence
analysis. This idea goes back to Bregman \cite{Bregman1967}.

\medskip\medskip

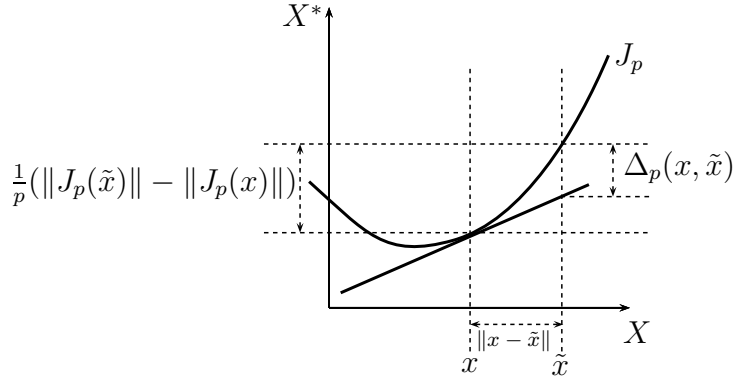
\begin{figure}[htp]
\centering
\scalebox{1} 
{
\begin{pspicture}(0,-2.4489062)(9.682813,2.4089062)
\rput(3.446771,-1.5910938){\psaxes[linewidth=0.024,tickstyle=bottom,labels=none,ticks=none,ticksize=0.10583333cm]{->}(0,0)(0,0)(4,4)}
\usefont{T1}{ptm}{m}{n}
\rput(7.5423436,-1.9210937){$X$}
\psbezier[linewidth=0.04](3.1809375,0.08890625)(3.9409375,-0.5910938)(4.1809373,-0.95109373)(5.1009374,-0.69109374)(6.0209374,-0.43109375)(6.7409377,0.76890624)(7.1609373,1.7689062)
\psline[linewidth=0.04cm](3.6009376,-1.3910937)(6.9009376,0.04890625)
\psline[linewidth=0.02cm,linestyle=dashed,dash=0.06cm 0.06cm](6.5409374,1.5889063)(6.5409374,-2.1710937)
\psline[linewidth=0.02cm,linestyle=dashed,dash=0.06cm 0.06cm](5.3209376,1.5889063)(5.3209376,-2.1710937)
\usefont{T1}{ptm}{m}{n}
\rput(7.432344,1.7189063){$J_p$}
\usefont{T1}{ptm}{m}{n}
\rput(6.52,-2.35){$\tilde{x}$}
\usefont{T1}{ptm}{m}{n}
\rput(5.32,-2.35){$x$}
\psline[linewidth=0.02cm,linestyle=dashed,dash=0.06cm 0.06cm](2.5809374,0.5889062)(7.6809373,0.5889062)
\psline[linewidth=0.02cm,linestyle=dashed,dash=0.06cm 0.06cm](6.5209374,-0.11109375)(7.7009373,-0.11109375)
\psline[linewidth=0.02cm,linestyle=dashed,dash=0.06cm 0.06cm,arrowsize=0.05291667cm 2.0,arrowlength=1.4,arrowinset=0.4]{<->}(7.2209377,0.5889062)(7.2209377,-0.11109375)
\usefont{T1}{ptm}{m}{n}
\rput(8.1,0.27890626){$\Delta_p(x, \tilde{x})$}
\psline[linewidth=0.02cm,linestyle=dashed,dash=0.06cm 0.06cm,arrowsize=0.05291667cm 2.0,arrowlength=1.4,arrowinset=0.4]{<->}(5.3209376,-1.8110938)(6.5609374,-1.8110938)
\usefont{T1}{ptm}{m}{n}
\rput(5.9,-2.025){\fontsize{8}{8}\selectfont $\|x - \tilde{x}\|$}
\psline[linewidth=0.02cm,linestyle=dashed,dash=0.06cm 0.06cm](2.5809374,-0.5910938)(7.6809373,-0.5910938)
\psline[linewidth=0.02cm,linestyle=dashed,dash=0.06cm 0.06cm,arrowsize=0.05291667cm 2.0,arrowlength=1.4,arrowinset=0.4]{<->}(3.0609374,0.5889062)(3.0609374,-0.5910938)
\usefont{T1}{ptm}{m}{n}
\rput(1.12,0.04){$\frac{1}{p} (\|J_p(\tilde{x})\| - \|J_p(x)\|)$}
\rput(3.1,2.3){$X^*$}
\end{pspicture}
}
\caption{Bregman distance, $\Delta_p$.}
\label{fig:1}
\end{figure}

%
%

\begin{definition}
Let $X$ be a uniformly smooth Banach space and $p > 1$. The Bregman
distance $\Delta_p (x, \cdot)$ of the convex functional $x\mapsto
\frac{1}{p}\|x\|^p$ at $x \in X$ is defined as
\begin{equation}\label{def:Bregman-distance}
   \Delta_p(x,\tilde{x}) =
       \frac{1}{p} \|\tilde{x}\|^p - \frac{1}{p} \|x\|^p
    - \langle J_p(x), \tilde{x} - x\rangle , \quad \tilde{x} \in X ,
\end{equation}
where $J_p$ denotes the duality mapping of $X$ with gauge function
$t \mapsto t^{p-1}$.
\end{definition}

\medskip\medskip

In the following theorem, we summarize some facts concerning the
Bregman distance and the relationship between Bregman distance and the
norm \cite{Alber1997, Alber1996,Butnariu2000,Xu1991}.

\medskip\medskip

\begin{theorem}\label{thm:Bregman-distance}
Let $X$ be a uniformly smooth and uniformly convex Banach space. Then,
for all $x, \tilde{x} \in X$, the following holds:
\begin{enumerate}[(a)]
\item
\begin{eqnarray}\label{eq:BDpq}
   \Delta_p(x,\tilde{x}) &=& \frac{1}{p}\|\tilde{x}\|^p
   - \frac{1}{p}\|x\|^p - \langle J_p(x), \tilde{x} \rangle + \|x\|^p
\\
   &=& \frac{1}{p}\|\tilde{x}\|^p + \frac{1}{q}\|x\|^p
       - \langle J_p(x), \tilde{x}\rangle .
\nonumber
\end{eqnarray}
\item $\Delta_p(x,\tilde{x}) \ge 0$ and $\Delta_p(x,\tilde{x}) = 0
  \Leftrightarrow x = \tilde{x}.$
\item $\Delta_p$ is continuous in both arguments.
\item The following statements are equivalent
\begin{enumerate}[(i)]
\item $\lim_{n \rightarrow \infty} \|x_n - x\| = 0, $
\item $\lim_{n \rightarrow \infty} \Delta_p(x_n , x) = 0 ,$
\item $\lim_{n \rightarrow \infty} \|x_n\| = \|x\|$ and $ \lim_{n
  \rightarrow \infty} \langle J_p(x_n), x\rangle = \langle J_p(x),
  x\rangle$.
\end{enumerate}
\item If $X$ is $p$-convex, there exists a constant $C_p > 0$ such
  that
\begin{equation}\label{Bregman-norm-rela1}
   \Delta_p(x,\tilde{x}) \ge \frac{C_p}{p} \|x-\tilde{x}\|^p .
\end{equation}
\item If $X^*$ is $q$-smooth, there exists a constant $G_q > 0$ such
  that
\begin{equation}\label{Bregman-norm-rela2}
   \Delta_q(x^*,\tilde{x}^*) \le \frac{G_q}{q} \|x^*-\tilde{x}^*\|^q ,
\end{equation}
for all $x^*, \tilde{x}^*\in X^*.$
\end{enumerate}
\end{theorem}

\medskip\medskip

\begin{remark}
The Bregman distance $\Delta_p$ is similar to a metric, but, in
general, does not satisfy the triangle inequality nor symmetry. In a
Hilbert space, $\Delta_2(x,\tilde{x}) = \frac{1}{2}
\|x-\tilde{x}\|^2.$
\end{remark}

\subsection{Landweber iteration}
\label{sec:2-2}

In this subsection, we introduce an iterative method for minimizing the functional \begin{equation}
\Phi(x) = \frac{1}{p} \|F(x) - y\|^p.
\end{equation}
The iterates $\{x_k\}$ are generated with  the steepest descent flow given by
\begin{equation}
\partial\Phi^{(k)}(x_k) = DF(x_k)^*j_p(F(x_k) - y)  .
\end{equation}
To be more precisely, we study the iterative method in Banach spaces,
\begin{equation}\label{Landweber-iter-Banach}
\begin{array}{rl}
   J_p(x_{k+1}) = & J_p(x_{k}) - \mu DF(x_k)^* j_p(F(x_k) - y)
           ,
\\[0.2cm]
   x_{k+1} = & J_q^*(J_p(x_{k+1})) ,
\end{array}
\end{equation}
where $J_p: X \rightarrow X^*$, $J_q^*: X^* \rightarrow X$ and $j_p:
Y\rightarrow Y^*$ denote duality mappings in corresponding spaces. When $X$ and $Y$ are Hilbert spaces and $p = 2$, this reduces to the Landweber iteration in Hilbert spaces
\begin{equation}\label{Landweber-iter-Hilbert}
   x_{k+1} = x_{k} - \mu DF(x_k)^* (F(x_k) - y).
\end{equation}
If $F$ is a linear operator, the iteration (1.3) coincides with
Landweber's original algorithm. We
specify $\mu$ below. Equation
(\ref{Landweber-iter-Banach}) defines a sequence $(x_k)$.

If $F(x^\dagger) = y$, the so-called tangential cone condition
\cite{Kaltenbacher2009},
\begin{equation}
   \| F(x) - F(\tilde{x}) - DF(x) (x - \tilde{x}) \|
   \le c_{\mathrm{TC}}\
            \| F(x) - F(\tilde{x}) \| \quad
   \forall x,\tilde{x} \in \mathcal{B}^{\Delta}_{\rho}(x^{\dagger}) ,
\end{equation}
for some $0 < c_{\mathrm{TC}} < 1$, is crucial to obtain convergence
of $(x_k)$ to $x^\dagger$
\cite{Hein2010,Hofmann2007,Kaltenbacher2009};
$\mathcal{B}^{\Delta}_{\rho}(x^{\dagger}) = \{ x \in
X\ |\ \Delta_p(x,x^{\dagger}) \le \rho \} \subset \mathcal{D}(F)$. A
source condition controls the convergence rate. Here, we study
convergence and convergence rates in relation to a single, alternative
condition replacing the tangential cone and source conditions, namely,
H\"{o}lder type stability,
\begin{equation*}
   \Delta_p(x,\tilde{x}) \le
   C_F^p \|F(x) - F(\tilde{x})\|^{\frac{1 + \varepsilon}{2}p} \quad
   \forall x,\tilde{x} \in \mathcal{B}^{\Delta}_{\rho}(x^{\dagger}) ,
\end{equation*}
for some $\varepsilon \in (0,1]$. With the Fr\'echet differentiability of $F$ and the Lipschitz continuity of $DF$, this condition implies the
  tangential cone condition, and, hence, convergence is guaranteed;
  however, it also implies a certain convergence rate.

\section{Convergence rate and radius of convergence -- Hilbert spaces}
\label{sec:3}

In this section, we assume that $X$ and $Y$ are Hilbert spaces. Then
the mappings $J_p, j_p$ and $J^*_q$ are identity mappings. Let
$\mathcal{B}_{\rho}(x_0)$ denote a closed ball centered at $x_0$ with
radius $\rho$, such that $\mathcal{B} = \mathcal{B}_{\rho'}(x_0)
\subset \mathcal{D}(F)$, $\rho' > \rho$. As before, let $x^\dag$
generate the data $y$, that is
\begin{equation}\label{operator-equ}
   F(x^{\dag}) = y .
\end{equation}
We assume that $x^{\dag} \in \mathcal{B}_{\rho}(x_0)$.

\medskip\medskip

\begin{assumption}
\label{assumption:forward-operator-Hilbert-Holder}
\begin{enumerate}[(a)]
\item The Fr\'{e}chet derivative, $DF$, of $F$ is Lipschitz continuous
  locally in $\mathcal{B}$ and
  \begin{equation}
   \| DF(x) - DF(\tilde{x}) \| \le L \| x - \tilde{x} \|
   \quad \forall x, \tilde{x} \in \mathcal{B} .
\end{equation}
\item $F$ is weakly sequentially closed, that is,
\begin{equation*}
\left. \begin{array}{rl}
      & x_n \rightharpoonup x ,
 \\
      & F(x_n) \rightarrow y
      \end{array} \right\} \Rightarrow
      \left\{ \begin{array}{rl}
      & x \in \mathcal{D}(F) ,
\\
      & F(x) = y .
\end{array}\right.
\end{equation*}
\item The inversion has the uniform H\"{o}lder type stability, that is,
  there exists a constant, $C_F > 0$, such that
\begin{equation}\label{stab-Hilbert-Holder}
   \frac{1}{\sqrt{2}}\|x - \tilde{x}\|
   \le C_F \|F(x) - F(\tilde{x})\|^{\frac{1+\varepsilon}{2}} \quad
           \forall x, \tilde{x} \in \mathcal{B}
\end{equation} for some $\varepsilon\in (0,1]$
\end{enumerate}
\end{assumption}

\medskip \medskip

In the remainder of this section, we discuss the convergence criterion
and convergence rate for the Landweber iteration
(\ref{Landweber-iter-Hilbert}).

\medskip\medskip
\medskip\medskip

\begin{theorem}
Assume there exists a solution $x^\dag$ to (\ref{operator-equ}) and
that Assumption~\ref{assumption:forward-operator-Hilbert-Holder}
holds. Furthermore, assume that
\begin{equation}\label{eq:DFb-h}
   \| DF(x) \| \le \hat{L} \quad \forall x \in \mathcal{B}.
\end{equation}
Let the positive stepsize $\mu$ be such that
\begin{equation}\label{eq:mucond-h}
\begin{array}{rl}
   \mu < & \dst{\frac{1}{\hat{L}^2}} ,
\\
   \dst{\mu (1 -
   \mu \hat{L}^2)} < & 2^{\frac{2}{1+\varepsilon}} C_F^{\frac{4}{1+\varepsilon}}
\end{array}
\end{equation}
Let
\begin{equation*}
\rho = \frac{1}{2}( 2L \hat{L}^\varepsilon C_F^2 )^{-2/\varepsilon}.
\end{equation*}
If
\begin{equation}\label{converge-radius}
  \frac{1}{2}\|x_0 - x^\dag\|^2 \le \rho ,
\end{equation}
then the iterates satisfy
\begin{equation}
   \frac{1}{2}\|x_k - x^\dag\|^2 \le \rho, \quad
   k =1,2,\ldots
\end{equation}
and $x_k \rightarrow x^\dag \mbox{ as } k \rightarrow \infty.$
Moreover, let
\begin{equation}
   c = \frac{1}{2} \mu (1 - \mu \hat{L}^2)
                 C_F^{-\frac{4}{1+\varepsilon}} ;
\end{equation}
from (\ref{eq:mucond-h}), it follows that $0 < c < 1$.  The
convergence rate is given by
\begin{equation}\label{convergence-rate}
   \frac{1}{2}\| x_k - x^\dag \|^2 \le
    \rho (1-c)^{k},
\end{equation}
if $\varepsilon = 1$. For $\varepsilon \in (0,1)$, the convergence
rate is given by
\begin{equation}\label{convergence-rate-h}
   \frac{1}{2}\| x_k - x^\dag \|^2 \le
   \left( ck \frac{1 - \varepsilon}{1+\varepsilon}
   + \rho^{-\frac{1 - \varepsilon}{1+\varepsilon}}
   \right)^{-\frac{1 + \varepsilon}{1 - \varepsilon}} ,\quad
   k = 0,1,\dots .
\end{equation}
\end{theorem}

\medskip\medskip

The proof is a special case of the Banach space setting,
cf.~Theorem~\ref{thm:Banach-converge}; see Section~\ref{sec:4}. The
convergence is sublinear if $0<\varepsilon < 1$ and the speed up as
$\varepsilon \rightarrow 1$ relates to the fact that it switches to a
linear convergence.

For the critical index $\varepsilon = 0$, that is, the power in the
right-hand side of the stability inequality
(\ref{stab-Hilbert-Holder}) equals to $\frac{1}{2}$, we need to invoke
an assumption on the stability constant $C_F$ to arrive at the
convergence and convergence rate results. An interesting by-product is
that the convergence radius only depends on the radius within which
the H\"{o}lder stability (\ref{stab-Hilbert-Holder}) holds. Hence, if
the forward operator $F$ satisfies (\ref{stab-Hilbert-Holder})
globally, then we obtain a global convergence and convergence rate
result.

\medskip\medskip

\begin{theorem}\label{thm:critical-index}
Assume there exists a solution $x^\dag$ to (\ref{operator-equ}) and
that Assumption~\ref{assumption:forward-operator-Hilbert-Holder} holds
with $\varepsilon = 0$. Furthermore, assume that
\begin{equation}
   \| DF(x) \| \le \hat{L} \quad \forall x \in \mathcal{B}.
\end{equation}
Let the stability constant $C_F$ and the positive stepsize $\mu$
satisfy that
\begin{equation}\label{eq:mucond-h-cri}
 \mu \hat{L}^2 +  2 L C_F^2 < 2 .
\end{equation}
Then the iterates satisfy
\[
x_k \rightarrow x^\dag \mbox{ as } k \rightarrow \infty.
\]
Moreover, let
\begin{equation}
   c =\frac{\mu}{4}(-2 + \mu \hat{L}^2 +  2L C_F^2) C_F^{-4} .
\end{equation}
The convergence rate is given by
\begin{equation}\label{convergence-rate-cr}
   \frac{1}{2}\| x_k - x^\dag \|^2 \le
        (2\| x_0 - x^\dag \|^{-2} + ck)^{-1} .
\end{equation}
\end{theorem}

\medskip\medskip

The proof is again a special case of the Banach space setting,
cf.~Theorem~\ref{thm:Banach-converge}; see Section~\ref{sec:4}.

\medskip\medskip

\begin{remark}
The convergence radius condition (\ref{converge-radius}) on the
starting point $x_0$ may be replaced by a convergence radius condition
on the starting simulated data $F(x_0)$,
\begin{equation}\label{converge-radius-data-H}
   \|F(x_0) - y\|^{ 1 + \varepsilon } \le \rho C_F^{-2} .
\end{equation}
In fact, with the aid of the stability inequality
(\ref{stab-Hilbert-Holder}), (\ref{converge-radius}) follows from
(\ref{converge-radius-data-H}).
\end{remark}

\section{Convergence rate and radius of convergence -- Banach spaces}
\label{sec:4}

In this section, we discuss the convergence and convergence rate of
the Landweber iteration (\ref{Landweber-iter-Banach}) in
Banach spaces. Let $\mathcal{B}_{\rho}(x_0)$ denote a closed ball
centered at $x_0$ with radius $\rho$, and $\mathcal{B} =
\mathcal{B}^{\Delta}_{\rho}(x^{\dag})$ denote a ball with respect to
the Bregman distance centered at some solution $x^{\dag}$. We assume
that $\mathcal{B}^{\Delta}_{\rho}(x^{\dag}) \subset \mathcal{D}(F)$.

\medskip\medskip

\begin{assumption}\label{assumption:forward-operator}
\begin{enumerate}[(a)]
\item The Fr\'{e}chet derivative, $DF$, of $F$ is Lipschitz
  continuous locally in $\mathcal{B}$ and
  \begin{equation}
   \| DF(x) - DF(\tilde{x}) \| \le L \| x - \tilde{x} \|
   \quad \forall x, \tilde{x} \in \mathcal{B}.
\end{equation}
\item $F$ is weakly sequentially closed, that is,
\begin{equation*}
   \left. \begin{array}{rl}
      & x_n \rightharpoonup x ,
\\
      & F(x_n) \rightarrow y
   \end{array} \right\} \Rightarrow
   \left\{ \begin{array}{rl}
      & x \in \mathcal{D}(F) ,
\\
      & F(x) = y .
   \end{array} \right.
\end{equation*}
\item The inversion has the uniform H\"{o}lder type stability, that is,
  there exists a constant $C_F > 0$ such that
\begin{equation}\label{stab-Banach}
   \Delta_p(x,\tilde{x}) \le
   C_F^p \|F(x) - F(\tilde{x})\|^{\frac{1+\varepsilon}{2}p}
   \quad \forall x, \tilde{x} \in \mathcal{B} ,
\end{equation}
for some $\varepsilon \in (0,1]$.
\end{enumerate}
\end{assumption}

\medskip\medskip

\begin{remark}
Note that the nonemptyness of the interior (with respect to norm) of
$\mathcal{D}(F)$ is sufficient for $\mathcal{B} \subset
\mathcal{D}(F)$.
\end{remark}

\medskip\medskip

\begin{remark}
With the assumption that $X$ is $p$-convex, (\ref{stab-Banach}) with
(\ref{Bregman-norm-rela1}) implies the regular notion of H\"{o}lder
stability in norm.
\end{remark}

\medskip\medskip

\begin{remark}
Under the Lipschitz type stability assumption, that is,
(\ref{stab-Banach}) with $\varepsilon =1$, we have that
\begin{equation*}
\begin{array}{rl}
\langle J_p(x^\dag), x - x^\dag\rangle  \le & \|x^\dag\|^{p-1} \|x - x^\dag\|
\\[0.2cm]
\le & C \Delta_p(x, x^\dag)^{1/p}
\\[0.2cm]
\le & CC_F \|F(x) - F(x^\dag)\|, \quad \forall x\in \mathcal{B}
\end{array}
\end{equation*}
for some constant $C>0$. It has been shown in \cite{Scherzer2009} that
this implies the source condition,
\begin{equation*}
   J_p(x^\dag) = DF(x^\dag)^* \omega
\end{equation*}
for some $\omega$ satisfying $\|\omega\| \le 1$.
\end{remark}

\medskip\medskip

\begin{theorem}\label{thm:Banach-converge}
Let $Y$ be a general Banach space, and $X$ be a Banach space which is $p$-convex and $q$-smooth with $1/p + 1/q =1$. Assume there exists a solution $x^\dag$ to (\ref{operator-equ}) and
that Assumption~\ref{assumption:forward-operator} holds. Furthermore,
assume that
\begin{equation}\label{assumption-L}
   \| DF(x) \| \le \hat{L} \quad \forall x \in \mathcal{B}.
\end{equation}
Let the positive stepsize, $\mu$, be such that
\begin{equation}\label{mu-Banach}
\begin{array}{rl}
   \mu^{q-1} < & \dst{\frac{q}{2 G_q \hat{L}^q}} ,
\\
   \dst{\mu \left( \frac{1}{2} -
   \frac{G_q \hat{L}^q}{q} \mu^{q-1}\right)}
   < & C_F^{\frac{2p}{1+\varepsilon}}.
\end{array}
\end{equation}
Let
\begin{equation*}
\rho = \hat{L}^{-p}(L C_F^2)^{-\frac{p}{\varepsilon}} \left(  \frac{C_p}{p} \right)^{1+\frac{2}{\varepsilon}}.
\end{equation*}
If
\begin{equation}\label{converge-radius-Banach}
   \Delta_p(x_0,x^\dag) \le \rho ,
\end{equation}
then the iterates satisfy
\begin{equation}
   \Delta_p(x_k,x^\dag) \le \rho ,\quad k=1,2,\ldots
\end{equation}
and $\Delta_p(x_k,x^\dag) \rightarrow 0
   \mbox{ as } k \rightarrow \infty.$
Moreover, let
\begin{equation}   c = C_F^{-\frac{2p}{1 + \varepsilon}} \left(\frac{1}{2} \mu
              - \frac{G_q}{q} \mu^q \hat{L}^q \right). \end{equation}
 The convergence rate is given by
\begin{equation}\label{convergence-rate-Banach}
   \Delta_p(x_k,x^\dag) \le \rho (1 - c)^{k},
\end{equation} if $\varepsilon = 1$.
For $\varepsilon \in (0,1)$,
the convergence rate is given by
\begin{equation}\label{conv-rate-Ba-h}
\Delta_p(x_k, x^\dag) \le \left(ck \frac{1 - \varepsilon}{1+\varepsilon} + \rho^{-\frac{1 - \varepsilon}{1+\varepsilon}} \right)^{-\frac{1 + \varepsilon}{1 - \varepsilon}}, \quad k = 0,1,\dots .
\end{equation}

\end{theorem}

\begin{proof}
Using (\ref{eq:BDpq}) and (\ref{Duality-map-def}), we obtain, for
the sequence of residues,
\begin{equation}\label{iterative-bound}
\begin{array}{rl}
   & \Delta_p(x_{k+1},x^\dag)
\\[0.2cm]
   = & \dst{\Delta_p(x_{k},x^\dag) + \frac{1}{q} \left(
     \|x_{k+1}\|^p - \|x_{k}\|^p \right)
         - \langle J_p(x_{k+1}) - J_p(x_{k}) , x^\dag \rangle}
\\[0.2cm]
   = & \dst{\Delta_p(x_{k},x^\dag) + \frac{1}{q} \left(
     \|J_p(x_{k+1})\|^q  - \|J_p(x_{k})\|^q \right)}
\\[0.0cm]
     &   - \langle J_p(x_{k+1}) - J_p(x_{k}) , x^\dag \rangle .
\end{array}
\end{equation}
Applying (\ref{eq:BDpq}) and (f) of Theorem~\ref{thm:Bregman-distance}
with $x^* = J_p(x_{k+1})$ and $\tilde{x}^* = J_p(x_k)$, we get
\begin{equation}
\begin{array}{rl}
   & \displaystyle{\frac{1}{q} \left( \|J_p(x_{k+1})\|^q
                        - \|J_p(x_{k})\|^q \right)}
   \\[0.2cm]
   & \hspace*{1.5cm} \le \displaystyle{\frac{G_q}{q} \|J_p(x_{k+1}) - J_p(x_{k})\|^q
        + \langle J_p(x_{k+1}) - J_p(x_{k}) , x_k \rangle .}
\end{array}
\end{equation}
Substituting (\ref{Landweber-iter-Banach}) and using this inequality
in (\ref{iterative-bound}) yields
\begin{equation}\label{iterative-bound-2}
\begin{array}{rl}
   &\Delta_p(x_{k+1}, x^\dag) - \Delta_p(x_{k}, x^\dag)
\\
   \le & \frac{G_q}{q} \|\mu DF(x_k)^* j_p(F(x_k) - y)
             \|^q
\\ - &\langle \mu DF(x_k)^* j_p(F(x_k) - y) , x_k - x^\dag \rangle.
\end{array}
\end{equation}

We estimate each term in (\ref{iterative-bound-2}) separately. The
first term satisfies the estimate
\begin{equation}
 \dst{\frac{G_q}{q} \|\mu DF(x_k)^* j_p(F(x_k) - y)\|^q}
   \le  \dst{\frac{G_q}{q} \mu^q \hat{L}^q \|(F(x_k) - y)\|^p} .
\end{equation}
For the second term, we have that
\begin{equation}
\begin{array}{rl}
   & - \langle \mu DF(x_k)^* j_p(F(x_k) - y) , x_k - x^\dag \rangle
\\[0.2cm]
   = & -\mu  \langle j_p(F(x_k) - y) , DF(x_k)(x_k - x^\dag) \rangle
\\[0.2cm]
   = & -\mu ( \langle j_p(F(x_k) - y) , F(x_k) - y \rangle
\\[0.0cm]
     & \hspace*{2.0cm} - \langle j_p(F(x_k) - y) ,
             F(x_k) - y - DF(x_k)(x_k - x^\dag) \rangle ).
\end{array}
\end{equation}
Note that, by the fundamental theorem of calculus for Fr\'echet derivative, we obtain that
\begin{equation}
\|F(x_k) - y - DF(x_k)(x_k - x^\dag)\| \le \frac{L}{2}\|x_k - x^\dag\|^2.
\end{equation}
Then, using (\ref{Bregman-norm-rela1}) and
stability (c) of
Assumption~\ref{assumption:forward-operator}, we have
\begin{equation}
\begin{array}{rl}
   & - \langle \mu DF(x_k)^* j_p(F(x_k) - y) , x_k - x^\dag \rangle
\\[0.2cm]
   = & -\mu \|F(x_k) - y\|^p
\\[0.0cm]
     & \hspace*{2.0cm}
     + \mu \langle j_p(F(x_k) - y) ,
             F(x_k) - y - DF(x_k)(x_k - x^\dag) \rangle
\\[0.2cm]
   \le & \dst{-\mu \|F(x_k) - y\|^p
            + \frac{\mu}{2} L \|(F(x_k) - y)\|^{p-1} \|x_k - x^\dag\|^2}
\\[0.0cm]
   \le & \dst{-\mu \|F(x_k) - y\|^p
          + \frac{\mu}{2} L C_F^2 \left(\frac{p }{C_p}\right)^{2/p}
                        \|F(x_k) - y\|^{p+\varepsilon}} .
\end{array}
\end{equation}
Combining these estimates and using the notation
\begin{equation*}
   \gamma_k = \Delta_p(x_k,x^\dag) ,
\end{equation*}
we obtain
\begin{equation}\label{residual-bound}
\begin{array}{rl}
   \gamma_{k+1} - \gamma_k \le &
   \dst{\left(\frac{G_q}{q} \mu^q \hat{L}^q
               - \frac{1}{2} \mu \right) \|F(x_k) - y\|^p}
\\[0.2cm]
   - & \dst{\frac{1}{2} \mu \|F(x_k) - y\|^p
           + \frac{\mu}{2} L C_F^2 \left(\frac{p}{C_p}\right)^{2/p}
                      \|F(x_k) - y\|^{p+\varepsilon}}.
\end{array}
\end{equation}
We claim that
\begin{equation}
   \gamma_{k+1} = \Delta_p(x_{k + 1},x^\dag)
      \le \rho ,
\end{equation}
which we prove by induction. Assume that
\begin{equation}\label{induction-assumption-Banach}
   \Delta_p(x_m,x^\dag) \le \rho
\end{equation}
holds for $m = 0,1,\dots, k$. With the mean value inequality, it follows that
\begin{equation}
   \|F(x_m) - y\|^\varepsilon \le \hat{L}^\varepsilon\left(\frac{p}{C_p}\rho\right)^{\frac{\varepsilon}{p}} = \frac{1}{ L C_F^2
        (p / C_p)^{2/p}} ,\quad m = 0,1,2,\dots,k .
\end{equation}
Therefore,
\begin{equation}
   - \half \mu \|F(x_m) - y\|^p
         + \half \mu L C_F^2  (p / C_p)^{2/p} \|F(x_m) - y\|^{p+\varepsilon}
                  \le 0 ,
\end{equation}
$m = 0,1,2,\dots,k .$
Dropping this non-positive term, we obtain
\begin{equation}\label{inductive-inequality-Banach}
\begin{array}{rl}
   \gamma_{k+1} - \gamma_k \le &
   \dst{\left( \frac{G_q}{q} \mu^q \hat{L}^q
           - \frac{1}{2}\mu \right) \|F(x_k) - y\|^p}.
\end{array}
\end{equation}
Note that the term $\dst{\left( \frac{G_q}{q} \mu^q \hat{L}^q
           - \frac{1}{2}\mu \right) \|F(x_k) - y\|^p}$ is non-positive. We obtain that
\begin{equation}
   \Delta_p(x_{k+1} , x^\dag)
   \le \rho ,
\end{equation}
which establishes the claim.

Now, we return to (\ref{inductive-inequality-Banach}). By the
H\"{o}lder type stability (\ref{stab-Banach}), we have that
\begin{equation}\label{indu-inequ}
   \gamma_{k+1} \le \gamma_k - c \gamma_k^{\frac{2}{1+\varepsilon}}
\end{equation}
Note that, by the conditions on $\mu$, we have $0 < c$. By letting $k$
go to infinity on both sides of the above inequality, we conclude that
\begin{equation*}
   \gamma_k \rightarrow 0 \mbox{ as } k\rightarrow \infty .
\end{equation*}

In the remainder of the proof, we obtain the convergence rate. Note
that, with the choice (\ref{mu-Banach}) of $\mu$,
\begin{equation}\label{eq:cest}
   0 <  c < 1 .
\end{equation}
With $\varepsilon = 1$, we have
\begin{equation}\label{eq:convrateH}
   \gamma_{k+1} \le (1 - c) \gamma_k
\end{equation}
which expresses the convergence rate (\ref{convergence-rate-Banach}).

For the convergence rate with $\varepsilon \in (0,1)$, from (\ref{indu-inequ}), we obtain that
\[
\gamma_{k+1}^{- \frac{1 - \varepsilon}{1+\varepsilon}} \ge \gamma_k^{- \frac{1 - \varepsilon}{1+\varepsilon}} (1 - c \gamma_k^{\frac{1 - \varepsilon}{1+\varepsilon}})^{- \frac{1 - \varepsilon}{1+\varepsilon}}.
\]
Noting that
\[
(1 - x)^{- \frac{1 - \varepsilon}{1+\varepsilon}} \ge 1 + \frac{1 - \varepsilon}{1+\varepsilon} x \quad \forall x\in (0,1),
\]
we have that
\[
\gamma_{k+1}^{- \frac{1 - \varepsilon}{1+\varepsilon}} \ge \gamma_k^{- \frac{1 - \varepsilon}{1+\varepsilon}} + c \frac{1 - \varepsilon}{1+\varepsilon}.
\]
It follows that
\[
\gamma_k \le \left(ck \frac{1 - \varepsilon}{1+\varepsilon} + \gamma_0^{-\frac{1 - \varepsilon}{1+\varepsilon}} \right)^{-\frac{1 + \varepsilon}{1 - \varepsilon}} \le \left(ck \frac{1 - \varepsilon}{1+\varepsilon} + \rho^{-\frac{1 - \varepsilon}{1+\varepsilon}} \right)^{-\frac{1 + \varepsilon}{1 - \varepsilon}}, \quad k = 0,1,\dots .
\]
\end{proof}
%
%

\medskip \medskip

For the critical index $\varepsilon =0$, we obtain

\medskip \medskip

\begin{theorem}\label{thm:critical-index-Banach}
Let $Y$ be a general Banach space, and $X$ be a Banach space which is $p$-convex and $q$-smooth with $1/p + 1/q =1$. Assume there exists a solution $x^\dag$ to (\ref{operator-equ}) and
that Assumption~\ref{stab-Banach} holds with $\varepsilon =
0$. Furthermore, assume that
\begin{equation}
   \| DF(x) \| \le \hat{L} \quad \forall x \in \mathcal{B},
\end{equation}
and that the stability constant $C_F$ and the positive stepsize $\mu$
satisfy the inequality
\begin{equation}\label{eq:mucond-h-cri-Ba}
 \mu^{q-1} < \frac{q}{G_q \hat{L}^q} \left(1 - \frac{1}{2} L C_F^2\left(\frac{p}{C_p}\right)^{\frac{2}{p}}\right).
\end{equation}
Then the iterates satisfy
\[
\Delta_p (x_k,x^\dag) \rightarrow 0 \mbox{ as } k \rightarrow \infty.
\]
Moreover, let
\begin{equation}   c = \mu\left(\frac{G_q}{q} \mu^{q-1} \hat{L}^q - 1 + \frac{1}{2}LC_F^2 \left(\frac{p}{C_p}\right)^{\frac{2}{p}}\right)C_F^{-2p}. \end{equation}
The convergence rate is given by
\begin{equation}\label{convergence-rate-cr-Ba}
   \Delta_p (x_k,x^\dag) \le
          (\Delta_p (x_0,x^\dag)^{-1} + ck)^{-1} .
\end{equation}
\end{theorem}
\medskip\medskip

\begin{proof}
Using (\ref{residual-bound}) in the proof of
Theorem~\ref{thm:Banach-converge} subject to the substitution
$\varepsilon = 0$, we obtain that
\[
   \gamma_{k+1} - \gamma_k \le
   \dst{\left(\frac{G_q}{q} \mu^q \hat{L}^q
               - \mu + \frac{\mu}{2} L C_F^2
   \left(\frac{p}{C_p}\right)^{2/p}\right) \|F(x_k) - y\|^p} .
\]
Note that, by (\ref{eq:mucond-h-cri-Ba}), the right-hand side of the
above inequality is non-positive and $0 <c <1$. Then, using the
H\"{o}lder type stability (\ref{stab-Banach}) with $\varepsilon =
0$, we have that
\begin{equation}
   \gamma_{k+1} \le  \gamma_k - c \gamma_k^2 .
\end{equation}
The convergence result and convergence rate
(\ref{convergence-rate-cr-Ba}) can be deduced by using the same
arguments as in the proof of Theorem~\ref{thm:Banach-converge}.
\end{proof}

\medskip\medskip

\begin{remark}
The H\"{o}lder type stability condition (\ref{stab-Hilbert-Holder}) or (\ref{stab-Banach}) is implied by a lower bound of the Fr\'{e}chet derivative $DF$. More precisely, if, there exists a constant $C$ such that
\[
\|DF(x)\left(\frac{x -x^\dag}{\|x - x^\dag\|} \right)\| \ge C \|x - x^\dag\|^{1 - \alpha} \quad \forall x\in B_r(x^\dag) \cap \mathcal{D}(F),
\]
for some $\alpha \in (0,1]$ and $r$ sufficiently small,
then, by combining this and
\[
\|F(\tilde{x}) - F(x) - DF(x)(\tilde{x} - x)\| \le \frac{L}{2}\|\tilde{x} - x\|^2 \quad \forall x, \tilde{x} \in \mathcal{D}(F)
\]
we obtain that
\[
\|x - x^\dag\| \le C_F \|F(x) - F(x^\dag)\|^{\frac{1}{2 - \alpha}} \quad \forall x\in B_r(x^\dag) \cap \mathcal{D}(F),
\]
for some constant $C_F$ depending on $C$ and $L$. The ill-posedness of
many inverse problems indicates that in general it is impossible to
obtain a lower bound for $DF$. If one projects the forward operator
$F$ properly, an estimate for the lower bound of $DF$ could be
obtained. Under various conditions, the lower bound for $DF$ has been
investigated in the analysis of inverse problems. For example, see
Calder\'{o}n \cite{Calderon1980}, Somersalo, Cheney \& Isaacson
\cite{Somersalo1992}, and Dobson \cite{Dobson1992} for the electrical
impedance tomography problem, Bao, Chen \& Ma \cite{Bao2000} for the
inverse medium problem associated with the Helmholtz equation, and
Ammari \& Bao \cite{Ammari2001} for the inverse medium problem for
electromagnetic waves.
\end{remark}

\section{Example: Electrical Impedance Tomography}
\label{sec:6}

In this section, we discuss Calder\'{o}n's inverse problem, which
forms the mathematical foundation of the Electrical Impedance
Tomography (EIT) problem \cite{Calderon1980}. For a recent review, we
refer to Uhlmann \cite{Uhlmann2009}. We mention some key uniqueness
results, namely, by Kohn \& Vogelius \cite{Kohn1984,Kohn1985},
Sylvester \& Uhlmann \cite{Sylvester1987}, and Astala \&
P\"{a}iv\"{a}rinta \cite{Astala2006}. Here, we focus on results
pertaining to stability; see Alessandrini
\cite{Alessandrini1988,Alessandrini1991,Alessandrini2007}. In
particular, we relate to the work of Alessandrini \& Vesella
\cite{Alessandrini2005} and Beretta \& Francini
\cite{BerettaToappear}, who establish a Lipschitz type stability
estimate if the conductivity is piecewise constant on a finite number
of subdomains with jumps, for the real-valued and complex-valued
cases, respectively.

\subsection{The Dirichlet-to-Neumann map}
\label{sec:6-1}

Let $\Omega\subset \mathbb{R}^n$ be a bounded domain with smooth
boundary. The electrical conductivity of $\Omega$ is represented by a
bounded and positive function $\gamma(x)$. Given a potential $f\in
H^{1/2}(\partial\Omega)$ on the boundary, the induced potential $u\in
H^1(\Omega)$ solves the Dirichlet problem
\[
\left\{
\begin{array}{rl}
\nabla \cdot (\gamma \nabla u) = & 0, \quad \mbox{ in } \Omega\\
u = & f, \quad \mbox{ on } \partial\Omega.
\end{array}
\right.
\]
The Dirichlet-to-Neumann map, or voltage-to-current map, is given by
\[
\Lambda_{\gamma}(f) = \left.\left(\gamma \frac{\partial u}{\partial \nu}\right)\right|_{\partial \Omega},
\]
where $\nu$ denotes the unit outer normal vector to $\partial \Omega$.

The forward operator $F$ is defined by
\begin{equation}\label{forward-operator}
\begin{array}{rrl}
F : & X\subset L_+^\infty(\Omega) \rightarrow & \mathcal{L}(H^{1/2}(\partial\Omega),H^{-1/2}(\partial\Omega)),\\
& \gamma \mapsto & \Lambda_\gamma.
\end{array}
\end{equation}
The Fr\'{e}chet derivative $DF$ of $F$ at $\gamma = \gamma_0$ is given
by
\begin{equation}\label{Frechet_operator_def}
\begin{array}{rrl}
DF(\gamma_0) : & X \subset L^\infty(\Omega) &\rightarrow \mathcal{L}(H^{1/2}(\partial \Omega) ,H^{-1/2}(\partial \Omega) )\\
    & \delta \gamma &\mapsto DF(\gamma_0)(\delta \gamma),
\end{array}
\end{equation}
and $DF(\gamma_0)(\delta \gamma)$ is given by
\begin{equation}\label{linearized_operator_def}
\langle DF(\gamma_0)(\delta \gamma) \,f, g\rangle = \int_{\Omega}  \delta \gamma \nabla u \cdot \nabla v  \dd x, \quad  f,g \in H^{1/2}(\partial\Omega)
\end{equation}
where
\[
\left\{
\begin{array}{rl}
\nabla \cdot (\gamma_0 \nabla u) = \nabla \cdot (\gamma_0 \nabla v) = 0, \quad &\mbox{ in } \Omega,\\
u =  f,\quad  v =  g\quad & \mbox{ on } \partial\Omega.
\end{array}
\right.
\]
We note that $L^\infty(\Omega)$ is not a uniformly convex Banach
space. Furthermore, to get the H\"{o}lder type stability, the preimage
space needs to be reduced. We specify the proper space $X$ in
Subsection~\ref{sec:6-3}.

For $n = 2$, Astala and P\"{a}iv\"{a}rinta proved that
$\Lambda_\gamma$ uniquely determines $\gamma$ under the assumption
that $\gamma \in L^{\infty}(\Omega)$. For $n\ge 3$,
P\"{a}iv\"{a}rinta, Panchenko and Uhlmann \cite{Paivarinta2003} proved
the uniqueness under the assumption that $\gamma \in
W^{3/2,\infty}(\Omega)$.

\subsection{Lipschitz stability}
\label{sec:6-2}

It is possible to obtain Lipschitz type stability, essentially, by
assuming that $\gamma$ belongs to a particular finite dimensional
space.

We write $x=(x', x_n)$ where $x'\in \mathbb{R}^{n-1}$ for $n\ge
2$. With $B_R(x)$, $B'_R(x')$ and $Q_R(x)$ we denote respectively the
open ball in $\mathbb{R}^n$ centered at $x$ of radius $R$, the ball in
$\mathbb{R}^{n-1}$ centered at $x'$ of radius $R$ and the cylinder
$B'_R(x') \times (x_n - R, x_n + R)$. For simplicity of notation,
$B_R(0)$, $B'_R(0)$ and $Q_R(0)$ are denoted by $B_R$, $B'_R$ and
$Q_R$.

\medskip\medskip

\begin{definition}\label{def:Lip-portion}
Let $\Omega$ be a bounded domain in $\mathbb{R}^n$. We say that
$\partial\Omega$ is of Lipschitz class with constants $r_0, L>0$ if,
for any $P\in \partial\Omega$, there exists a rigid transformation of
coordinates such that $P = 0$ and
\[
\Omega \cap Q_{r_0} = \{(x',x_n)\in Q_{r_0} \mid x_n > \phi(x') \}
\]
where $\phi$ is a Lipschitz continuous function on $B_{r_0}'$ with $\phi(0)=0$ and
\[
\|\phi\|_{C^{0,1}(B_{r_0}')} \le L r_0 .
\]
\end{definition}

\medskip\medskip

\begin{definition}\label{def:Holder-class}
Let $\Omega$ be a bounded domain in $\mathbb{R}^n$. Given $\alpha
\in(0,1)$, we say that $\partial\Omega$ is of $C^{1,\alpha}$ class
with constants $r_0, L>0$ if, for any $P\in \partial \Omega$, there
exists a rigid transformation of coordinates such that $P = 0$ and
\[
\Omega \cap Q_{r_0} = \{(x',x_n)\in Q_{r_0} \mid x_n > \phi(x') \}
\]
where $\phi$ is a $C^{1,\alpha}$ function on $B_{r_0}'$ with $\phi(0)= | \nabla \phi(0)| = 0$ and
\[
\|\phi\|_{C^{1,\alpha}(B_{r_0}')} \le L r_0 .
\]
\end{definition}

\medskip\medskip

\begin{assumption}\label{assumption_domain}
$\Omega \subset \mathbb{R}^n$ is a bounded domain satisfying
\[
|\Omega| \le A |B_{r_0}|.
\]
Here and in the sequel $|\Omega|$ denotes the Lebesgue measure of
$\Omega$. We assume that $\partial \Omega$ is of Lipschitz class with
constants $r_0$ and $L$.
\end{assumption}

\medskip\medskip

\begin{assumption}\label{assumption_conductivity}
The conductivity $\gamma$ is a piecewise constant function of the form
\[
\gamma(x) = \sum_{j = 1}^N \gamma_j \chi_{D_j}(x),
\]
satisfying the ellipticity condition
\[
K^{-1} \le \gamma \le K
\]
for some constant $K$, where $\gamma_j, j=1,\dots N$ are unknown real
numbers and $D_j$ are known open sets in $\mathbb{R}^n$.
\end{assumption}

\medskip\medskip

\begin{assumption}\label{assumption_piece}
The $D_j, j=1,\dots, N$ are connected and pairwise non-overlapping
open sets such that $\cup_{j=1}^N \overline{D}_j = \overline{\Omega}$
and $\partial D_j$ are of $C^{1,\alpha}$ class with constants $r_0$
and $L$ for all $j = 1,\dots , N.$ We also assume that there exists
one region, say $D_1$, such that $\partial D_1 \cap \partial \Omega$
contains an open portion $\Sigma_1$ of $C^{1,\alpha}$ class with
constants $r_0$ and $L$. For every $j \in \{2, \dots, N\}$ there exist
$j_1, \dots , j_M \in \{1,\dots, N\}$ such that
\[
D_{j_1} = D_1, \quad D_{j_M} = D_j
\]
and, for every $k = 1, \dots, M$,
\[
\partial D_{j_{k - 1}} \cap \partial D_{j_k}
\]
contains a open portion $\Sigma_k$ of $C^{1,\alpha}$ class with
constants $r_0$ and $L$.
\end{assumption}

\medskip\medskip

Alessandrini and Vessella \cite{Alessandrini2005} establish the
following Lipschitz stability estimate

\medskip\medskip

\begin{theorem}\label{Lip-stability}
Let $\Omega$ satisfy Assumption~\ref{assumption_domain} and
$\gamma^{(k)}, k = 1, 2$ be two real piecewise constant functions
satisfying Assumption~\ref{assumption_conductivity} and $D_j,
j=1,\dots, N$ satisfying Assumption~\ref{assumption_piece}. Then
there exists a constant $C = C(n, r_0, L, A, ,K, N)$ such that
\begin{equation}\label{Lip_stab}
   \| \gamma^{(1)} - \gamma^{(2)}\|_{L^\infty(\Omega)}
   \le C \|\Lambda_{\gamma^{(1)}}
      - \Lambda_{\gamma^{(2)}}\|_{\mathcal{L}(
   H^{1/2}(\partial\Omega), H^{-1/2}(\partial\Omega))} .
\end{equation}
\end{theorem}

\subsection{Convergence}
\label{sec:6-3}

We verify that the assumptions of Section~\ref{sec:4} can be
satisfied. We specify our preimage space as
\begin{equation}\label{preimage-space}
X = \mbox{span}\{ \chi_{D_1},\dots , \chi_{D_N}\}
\end{equation}
equipped with $L^p$-norm, $p>1$. With the aid of this particular
basis of $X$, one can show that $F$ and $DF$ are Lipschitz
continuous. Moreover, assuming that $\gamma_1,\gamma_2$ satisfy
Assumption~\ref{assumption_conductivity} and $\Omega$ satisfies
Assumption~\ref{assumption_domain}, we have the following estimates:
\begin{equation}
\begin{array}{rl}
\|F(\gamma_1) - F(\gamma_2)\|_{\mathcal{L}(H^{1/2}(\Omega), H^{-1/2}(\Omega))} \le & C \|\gamma_1 - \gamma_2\|_{L^p(\Omega)},\\
\|DF\|_{\mathcal{L}(X,\mathcal{L}(H^{1/2}(\Omega), H^{-1/2}(\Omega)))} \le & \hat{L},\\
\|DF(\gamma_1) - DF(\gamma_2)\|_{\mathcal{L}(H^{1/2}(\Omega), H^{-1/2}(\Omega))} \le & L\|\gamma_1 - \gamma_2\|_{L^p(\Omega)},\\
\end{array}
\end{equation}
where $C$, $\hat{L}$ and $L$ depend on $\Omega$, $N$ and ellipticity
constant $K$. Furthermore, since $X$ is finite dimensional, the weak
topology is equivalent to the strong topology. Hence, $F$ is a weakly
sequentially closed operator.

Let $\Omega$ satisfy Assumption~\ref{assumption_domain}, preimage
space $X$ be defined by (\ref{preimage-space}) and $F$ be defined by
(\ref{forward-operator}). Assume that $y = F(\gamma^\dag)$ for some
$\gamma^\dag \in X$. Then
Assumption~\ref{assumption:forward-operator} and
(\ref{assumption-L}) of Theorem~\ref{thm:Banach-converge} are
satisfied. Hence the Landweber iteration
(\ref{Landweber-iter-Banach}) converges with convergence radius
given by (\ref{converge-radius-Banach}) and convergence rate given
by (\ref{convergence-rate-Banach}). Convergence of a regularized
Newton method for a finite dimensional EIT problem was proven by
Lechleitner and Rieder \cite{Lechleiter2008}. Their analysis, however,
is based on the tangential cone condition.

\section{Discussion}
\label{sec:7}

We discuss a Landweber iteration method for solving nonlinear operator
equations in both Hilbert and Banach spaces. Traditionally, the
gradient-type methods are often regarded as too slow for practical
applications. Provided that the nonlinearity of the forward operator
obeys a H\"{o}lder type stability, we could prove the convergence and
give a sublinear convergence rate. With a Lipschitz type stability,
the convergence rate switches to a linear one. Based on these
convergence rates, we anticipate that this Landweber iteration is a
valuable tool in solving inverse problems in both Hilbert and Banach
spaces. This also motivates the study of H\"{o}lder/Lipschitz type
stability in inverse problems to provide explicit reconstructions.

\section*{Acknowledgements}

The research in this paper was initialized at the Program on Inverse Problems and
Applications at MSRI, Berkeley, during the Fall of 2010. This research was supported in part by National Science Foundation grant CMG DMS-1025318, and in part by the members of the Geo-Mathematical Imaging Group at Purdue University. The work of O.S. was supported by the Austrian Science Fund (FWF) within the national research network Photo\-acoustic Imaging in Biology and Medicine, project S10505-N20.

\noindent

\def\cprime{$'$} \newcommand{\SortNoop}[1]{}

\section*{References}
\bibliographystyle{acm}
\bibliography{conv}

\end{document}